\documentclass[hidelinks]{siamart251216}
\usepackage{amssymb,mathtools,bm}
\usepackage{booktabs,array}
\usepackage{microtype}
\hypersetup{pdftitle={Phase-type representations and sharp bounds for SSP Runge--Kutta methods}}

\numberwithin{equation}{section}
\numberwithin{theorem}{section}
\newsiamremark{example}{Example}
\newsiamremark{remark}{Remark}

\newcommand{\E}{\mathbb{E}}
\newcommand{\Prob}{\mathbb{P}}
\newcommand{\Var}{\operatorname{Var}}
\newcommand{\Exp}{\operatorname{Exp}}

\newcommand{\PH}{\operatorname{PH}}
\newcommand{\Er}{\operatorname{Er}}
\newcommand{\one}{\mathbf e}
\newcommand{\trans}{\mathsf T}
\newcommand{\phT}{\mathbf T}
\newcommand{\phPi}{\boldsymbol\pi}
\newcommand{\phW}{\mathbf w}
\newcommand{\phExit}{\mathbf t}
\newcommand{\Generator}{\boldsymbol\Lambda}
\newcommand{\Green}{\mathbf U}
\newcommand{\zero}{\mathbf 0}
\newcommand{\ind}{\mathbf 1}
\headers{Phase-type representations and sharp bounds}{Z. Weng, H. Yu, and Q. Zhu}
\title
{Phase-type representations
for Strong Stability Preserving Runge--Kutta methods}
\author{Zhanyang Weng\thanks{Key Laboratory of Information Science of Electromagnetic Waves (MoE), Fudan University. (\email{zyweng26@m.fudan.edu.cn})}
\and Haijun Yu\thanks{Academy of Mathematics and Systems Science,
Chinese Academy of Sciences, Beijing 100190, China
(\email{hyu@lsec.cc.ac.cn}, \email{zhuquanfan@amss.ac.cn}).}
\and Quanfan Zhu\footnotemark[2]}
\date{}

\begin{document}
\maketitle
\begin{abstract}
We develop a phase-type (PH) representation of strong stability preserving (SSP) Runge--Kutta (RK)  methods. This representation provides a probabilistic interpretation of the RK coefficients and translates the chain order conditions into absorption-time moment identities. Using moment ratio bounds for PH distributions, we prove the conjectured optimal SSP coefficients for general implicit RK methods of orders two and three. The equality characterizations also establish uniqueness of the optimal RK tableaux.
\end{abstract}

\section{Introduction}
\label{sec:introduction}
Strong stability preserving (SSP) Runge--Kutta (RK)  methods preserve stability properties satisfied by the forward Euler method with respect to a given norm, seminorm, or convex functional, under a suitable time-step restriction. 
For overviews of  SSP methods, see \cite{gottlieb2001strong, gottlieb2009highorder, gottlieb2011sspbook, izzo2022ssp}.
The SSP coefficient measures the allowable time-step size relative to that of forward Eule. A central problem  in the development of SSP methods 
is to construct schemes with the largest possible SSP coefficient. 

The Shu--Osher representation \cite{shu1988tvd} expresses the explicit RK scheme as convex combinations of the previous solution and forward Euler steps. This approach was later extended to general RK methods in \cite{ferracina2005extension}, leading to the generalized Shu--Osher representation. 
The SSP property is closely related to the absolute-monotonicity theory developed by Kraaijevanger \cite{kraaijevanger1991contractivity}. 
For an irreducible Runge--Kutta method, the optimal SSP coefficient equals the radius of absolute monotonicity of its Butcher tableau \cite{ferracina2004stepsize,gottlieb2009highorder}. This connection makes the absolute-monotonicity conditions a useful tool for studying  SSP property.

The optimality of several classical explicit SSP Runge--Kutta methods has been established in \cite{gottlieb1998tvd, spiteri2002optimal}. For implicit methods, Ferracina and Spijker \cite{ferracina2008strong} studied singly diagonally implicit Runge--Kutta (SDIRK) methods, while Ketcheson et al.\ \cite{ketcheson2009optimal} investigated the larger class of general implicit Runge--Kutta methods. Let $C_{s,p}$ denote the optimal SSP coefficient over $s$-stage $p$-order Runge--Kutta methods. These studies led to the following conjectures for the optimal SSP coefficients at orders two and three:
\begin{equation}
 C_{s,2}=2s\quad(s\geq1),\qquad
 C_{s,3}=s-1+\sqrt{s^2-1}\quad(s\geq2).
 \label{eq:mainradii}
\end{equation}
In \cite{kocsis2014optimal}, the authors determined the optimal coefficient for second-order diagonally implicit Runge--Kutta (DIRK) methods. 
Prior to the present work, optimality over general implicit tableaux remained open at both orders.

Probability  also provides tools for studying SSP optimization. 
Ait-Haddou \cite{aithaddou2021computation} related the optimization of explicit linear SSP methods to moment-matching for the Poisson distribution. 

Inspired by this probabilistic approach and the algebraic transformation in \cite{kocsis2014optimal}, we explore a connection between SSP methods and phase-type (PH) distributions, which describe the absorption time of  finite-state continuous-time Markov chains (CTMCs) \cite{bladt2017matrix}.   Under the absolute-monotonicity conditions, every SSP Runge--Kutta tableau defines a PH representation. This representation also gives a probabilistic interpretation of the RK coefficients and translates the chain order conditions into moment equalities on the absorption time.

The classical minimum-variance theorem \cite{aldous1987least} gives a sharp lower bound on the second-order moment ratio of a PH distribution. We extend the  third-order moment ratio bound for acyclic PH distributions \cite{osogami2006closed} to general PH representations. We also establish shch moment bounds of arbitrary order for acyclic PH distributions.

Combining these moment inequalities with the RK--PH correspondence, we prove the conjectures \eqref{eq:mainradii} for general implicit tableaux and determine the uniqueness of the optimal tableaux up to simultaneous stage permutations at both orders.

The remainder of this paper is organized as follows. Section~\ref{sec:preliminaries} introduces SSP Runge--Kutta methods and PH distributions. Section~\ref{sec:transform} develops the RK--PH correspondence. Section~\ref{sec:variance} establishes moment ratio bounds for PH distributions. 
Section~\ref{sec:applications} applies these bounds to prove the optimal SSP coefficients and characterize the attaining tableaux. Section~\ref{sec:conclusion} concludes the paper.

\section{Preliminaries}
\label{sec:preliminaries}

\subsection{Runge--Kutta methods and order conditions}
\label{sec:rkbackground}
For an autonomous differential equation $u'=f(u)$,
an $s$-stage Runge--Kutta method with step size $h$ is
\begin{equation}
\begin{aligned}
 U_i&=u_n+h\sum_{j=1}^s a_{ij}f(U_j),\qquad i=1,\ldots,s,\\
 u_{n+1}&=u_n+h\sum_{i=1}^s b_i f(U_i).
\end{aligned}
\label{eq:rkdefinition}
\end{equation}
Its Butcher tableau is specified by $A=(a_{ij})\in\mathbb R^{s\times s}$,
$b\in\mathbb R^s$, and $c=A\one$, where $\one=(1,\ldots,1)^\trans$ is the all-ones column vector.

If $A$ is strictly lower triangular,
that is $a_{ij}=0$ whenever $j\geq i$, then \eqref{eq:rkdefinition}
is called an explicit Runge--Kutta method. Otherwise it is called an
implicit Runge--Kutta method. In particular, if $A$ is lower triangular
with at least one nonzero diagonal entry, the method is called a diagonally
implicit Runge--Kutta (DIRK) method.

Classical order at least $p$ means that, starting from an exact value,
the one-step error is $O(h^{p+1})$ for every sufficiently smooth
equation.  

The first-order condition $b^\trans\one=1$ is also called the consistency condition. The chain order condition is
\begin{equation}
 b^\trans A^{j-1}\one=\frac1{j!},\qquad j=1,2,3,\ldots
\label{eq:rkchainconditions}
\end{equation}
 Here \(\one\) denotes the \(s\times1\) vector of ones. For $p\geq3$, the chain conditions are only a subset of the full order conditions.

For a rooted tree $\mathfrak t=[\mathfrak t_1,\ldots,\mathfrak t_q]$,
let $|\mathfrak t|$ denote its number of vertices. Define the elementary
stage-weight vector $F_{\mathfrak t}$ and the tree factorial
$\gamma(\mathfrak t)$ recursively by
$F_\bullet=\one$, $\gamma(\bullet)=1$, and
\[
 F_{\mathfrak t}=(AF_{\mathfrak t_1})\cdot\ldots\cdot(AF_{\mathfrak t_q}),
 \qquad
 \gamma(\mathfrak t)=|\mathfrak t|\prod_{\ell=1}^q\gamma(\mathfrak t_\ell).
\]
Here $\bullet$ is the single-vertex tree and the bracket joins the
subtrees to a new root. The classical order conditions
\cite{butcher2016numerical} are equivalently
\begin{equation}
 b^\trans F_{\mathfrak t}=\frac{1}{\gamma(\mathfrak t)},
 \qquad |\mathfrak t|\leq p.
\label{eq:rkorderconditions}
\end{equation}

For $p\leq4$, Table~\ref{tab:orderconditions} lists the necessary
and sufficient algebraic conditions
\cite{butcher2016numerical}.

\begin{table}[htbp]
\centering
\caption{Classical order conditions for Runge--Kutta methods up to order four.}
\label{tab:orderconditions}
\renewcommand{\arraystretch}{1.4}
\begin{tabular}{@{}cl@{}}
\toprule
Order & Additional conditions \\
\midrule
1 & $b^\trans\one=1$ \\
2 & $b^\trans c=\dfrac12$ \\
3 & $b^\trans(c\cdot c)=\dfrac13,\qquad b^\trans Ac=\dfrac16$ \\[1.1ex]
4 & \begin{tabular}{@{}l@{}}
    $b^\trans(c\cdot c\cdot c)=\dfrac14,\qquad b^\trans(c\cdot Ac)=\dfrac18$ \\[1ex]
    $b^\trans A(c\cdot c)=\dfrac1{12},\qquad b^\trans A^2c=\dfrac1{24}$
    \end{tabular} \\
\bottomrule
\end{tabular}
\end{table}

\subsection{SSP-RK methods}
\label{sec:sspbackground}
An SSP-RK method preserves any norm or
convex functional $\|\cdot\|$ that is nonincreasing under forward Euler for a
suitable step-size restriction. Suppose that the forward Euler method satisfies
\begin{equation}
 \|u+hf(u)\|\leq\|u\|,\qquad 0\leq h\leq h_{\mathrm{FE}},
\label{eq:FEbound}
\end{equation}
for every $u$ and some $h_{\mathrm{FE}}>0$. Then an RK method is  strong
stability preserving if
\begin{equation}
 \|U_i\|\leq\|u_n\|,\qquad
 \|u_{n+1}\|\leq\|u_n\|
\label{eq:SSPbound}
\end{equation}
hold for some constant $C>0$ and all $0\leq h\leq C h_{\mathrm{FE}}$. The 
largest admissible $C$ is the SSP coefficient.

This property is established by rewriting \eqref{eq:rkdefinition} in the
Shu--Osher form \cite{shuosher1988efficient,gottlieb2001strong}:
\begin{equation}
\begin{aligned}
U_i&=\alpha_{i0}u_n+\sum_{j=1}^s\alpha_{ij}\left(U_j+\frac{h}{r}f(U_j)\right),\qquad i=1,\ldots,s,\\
 u_{n+1}&=\gamma_0 u_n+\sum_{i=1}^s\gamma_i\left(U_i+\frac{h}{r}f(U_i)\right),
\end{aligned}
\label{eq:shuosher}
\end{equation}
where all coefficients are nonnegative and
$$\alpha_{i0}+\sum_{j=1}^s\alpha_{ij}=1,\qquad\gamma_0+\sum_{i=1}^s\gamma_i=1.$$
Since each equation is a convex combination of $u_n$ and forward Euler steps,
\eqref{eq:FEbound} implies \eqref{eq:SSPbound} whenever $0\leq h\leq r\,h_{\mathrm{FE}}$.

In \cite{kraaijevanger1991contractivity}, Kraaijevanger studied
the absolute monotonicity of Runge--Kutta methods.
A Runge--Kutta method with Butcher tableau $(A,b)$ is called
absolutely monotone at $-r(r>0)$ if and only if
\begin{equation}
\begin{aligned}
 &(I+rA)^{-1}\ \text{exists},\\
 &A(I+rA)^{-1}\geq\zero,\qquad
 rA(I+rA)^{-1}\one\leq\one,\\
 &b^\trans(I+rA)^{-1}\geq\zero,\qquad
 rb^\trans(I+rA)^{-1}\one\leq1.
\end{aligned}
\label{eq:K}
\end{equation}
Here, the inequalities are understood component-wise.

Let $(K)_r$ denote the conditions \eqref{eq:K}.
We call $r$ feasible if $(K)_r$ holds. The radius of absolute monotonicity of the complete tableau is defined by
\begin{equation}
 R(A,b)=\sup\bigl(\{0\}\cup\{r>0:(K)_r\text{ holds}\}\bigr).
\label{eq:RKradius}
\end{equation}

According to Ketcheson et al.\
\cite{ketcheson2009optimal}, the SSP coefficient
of an irreducible Runge--Kutta method is equal to its radius of absolute monotonicity.
Therefore, the SSP property of such a method can be studied through the algebraic conditions $(K)_r$.
Under the forward Euler assumption \eqref{eq:FEbound}, the Runge--Kutta
 method satisfies the SSP property \eqref{eq:SSPbound} whenever
\begin{equation}
 h\leq R(A,b)\,h_{\mathrm{FE}}.
\label{eq:SSPstepsize}
\end{equation}

\flushbottom
\subsection{Markov chains and phase-type distributions}
\label{sec:probbackground}

A phase-type distribution describes the time to absorption in a finite-state continuous-time Markov chain.

\begin{definition}
\label{def:ctmc}
Let $(X_t)_{t\geq0}$ be a stochastic process with right-continuous
sample paths on the finite state space $S=\{1,\ldots,N\}$.
Write $\mathcal F_t=\sigma(X_u:0\leq u\leq t)$ for its history
up to time $t$. The process is called a continuous-time Markov
chain (CTMC) if it satisfies the Markov property
\begin{equation}
 \Prob(X_{t+h}=j\mid\mathcal F_t)
 =\Prob(X_{t+h}=j\mid X_t),
 \qquad t,h\geq0,\quad j\in S.
\label{eq:markovproperty}
\end{equation}
The chain is called time-homogeneous if its transition probabilities
depend only on the elapsed time $h$. For any $t,h\geq0$ and $i,j\in S$, define
\begin{equation}
 p_{ij}(h)=\Prob(X_{t+h}=j\mid X_t=i)
 =\Prob(X_h=j\mid X_0=i).
\label{eq:timehomogeneity}
\end{equation}
\end{definition}
For a time-homogeneous CTMC, the transition matrices
$\mathbf P(t)=(p_{ij}(t))$ satisfy
$\mathbf P(0)=I$ and $\mathbf P(t+h)=\mathbf P(t)\mathbf P(h)$.
There exists an intensity matrix $\Generator=(\lambda_{ij})$ satisfying 
\begin{equation}
\lambda_{ij}\geq0\ (i\ne j),\qquad\lambda_{ii}=-\sum_{j\ne i}\lambda_{ij}
\label{eq:intensitymatrix}
\end{equation}
such that $\mathbf P(t)=e^{t\Generator}.$
For $i\ne j$, $\lambda_{ij}$ is the transition rate from $i$ to $j$.
The departure rate from state $i$ is $\lambda_i=-\lambda_{ii}$.
If $\lambda_i>0$, the holding time is $\Exp(\lambda_i)$.
If $\lambda_i=0$, state $i$ is called
\emph{absorbing}, meaning that once entered it is never left. Throughout the paper, we denote the absorbing state by $\Omega$.

Now we consider an absorbing CTMC on $\{1,\ldots,s,\Omega\}.$ Let the chain start from an arbitrary state in ${1,\ldots,s}$. We assume that it reaches the absorbing state in finite time with probability one.
Under this assumption, the total time spent in each of the states
$1,\ldots,s$ is finite almost surely. These states are called
\emph{transient states}.

The generator of the absorbing chain has the block form
\begin{equation}
 \Generator=\begin{pmatrix}\phT&\phExit\\\zero&0\end{pmatrix},
\label{eq:PHgenerator}
\end{equation}
where $\phT=(t_{ij})\in\mathbb R^{s\times s}$ describes the jumps
among the transient states and $\phExit$ gives the transition rates
to the absorbing state. Since the rows of $\Generator$ sum to zero,
$\phExit=-\phT\one$. 

\begin{definition}[Phase-type distribution]
\label{sec:phbackground}
Let $(X_t)_{t\geq0}$ be the absorbing CTMC  with
initial distribution $\phPi=(\pi_1,\ldots,\pi_s)$ on
$\{1,\ldots,s\}$. The absorption time
\[
 \tau=\inf\{t\geq0:X_t=\Omega\}
\]
has a \emph{phase-type distribution} with representation
$(\phPi,\phT)$, written $\tau\sim\PH(\phPi,\phT)$.
\end{definition}

According to \cite[Theorem~3.1.11]{bladt2017matrix},
a subintensity matrix $T$ is invertible if and only if
all states corresponding to $T$ are transient.
The following criterion  determines whether a given matrix
can serve as the transient generator in a PH representation.
\begin{proposition}
\label{prop:transience}
Let $\phT=(t_{ij})\in\mathbb R^{s\times s}$. Then $\phT$ can serve
as a transient subintensity matrix in a PH representation if and only if
the following conditions hold:

\begin{enumerate}
\renewcommand{\labelenumi}{\textup{(\arabic{enumi})}}
\item The off-diagonal entries of $\phT$ are nonnegative:
$t_{ij}\geq0\quad(i\ne j)$.
\item The row sums of $\phT$ are nonpositive:
$\phT\one\leq\zero$.
\item The matrix $\phT$ is nonsingular:
$\det\phT\ne0$.
\end{enumerate}
\end{proposition}

Usually,  the initial distribution satisfies $\phPi\one=1$.  In a few cases we also consider a
PH distribution with initial absorbing mass, which means
$\Prob(X_0=\Omega)=1-\phPi\one\geq0$. Hence, we define $\pi_0=1-\phPi\one$ and write
$\tau\sim\PH(\pi_0,\phPi,\phT)$. 
To distinguish between a fixed initial state $X_0=i$ and an
initial distribution $\phPi$, we write $\Prob_i,\E_i$ and
$\Prob_{\phPi},\E_{\phPi}$ for the corresponding probabilities and
expectations.

For $\tau\sim\PH(\phPi,\phT)$, the transient transition matrix is
$e^{u\phT}$, whose $(i,j)$ entry is
$\Prob_i(X_u=j,\tau>u)$. Summing over transient endpoints gives
\begin{equation}
 \Prob(\tau>u)=\phPi e^{u\phT}\one,\qquad
 F_\tau(u)=1-\phPi e^{u\phT}\one,\quad u\geq0.
\label{eq:PHcdf}
\end{equation}
Weighting these probabilities by the exit rates gives the density
\begin{equation}
 f_\tau(u)=\phPi e^{u\phT}\phExit,\qquad u>0.
\label{eq:PHdensity}
\end{equation}

For each transient state $j$, define the cumulative occupation time
\begin{equation}
 L_j=\int_0^\tau\ind_{\{X_t=j\}}\,dt,
 \qquad j=1,\ldots,s.
\label{eq:occupationtimes}
\end{equation}
Then $\tau=\sum_{j=1}^sL_j$. The \emph{Green matrix} collects the
expected occupation times:
\begin{equation}
 \Green=(-\phT)^{-1}=\int_0^\infty e^{u\phT}\,du,
 \qquad U_{ij}=\E_iL_j.
\label{eq:greenmatrix}
\end{equation}
For an initial distribution $\phPi$, these expectations form the row
$\phPi\Green$.
With reward rate $v_j\geq0$ in state $j$, the expected cumulative
reward is $\E_{\phPi}\sum_{j=1}^s v_jL_j=\phPi\Green v$,
where $v=(v_1,\ldots,v_s)^\trans$.

The absorption-time moments satisfy
\begin{equation}
 \E\tau^j=j!\phPi\Green^j\one,\qquad j\geq1.
\label{eq:PHmoments}
\end{equation}
In particular, $\E\tau=\phPi\Green\one$ and
$\Var(\tau)=2\phPi\Green^2\one-(\phPi\Green\one)^2$.

The Erlang law $\Er_s(\kappa)$ is the distribution of the sum of $s$ independent $\Exp(\kappa)$ variables.
It admits a sequential PH representation in which each transient state is visited exactly once before absorption.
To match the lower-triangular structure of the RK tableau, we label the transient states in decreasing order along the path.
In this case, the chain starts in state $s$ and follows
\[
 s\xrightarrow{\;\kappa\;}s-1\xrightarrow{\;\kappa\;}\cdots
 \xrightarrow{\;\kappa\;}1\xrightarrow{\;\kappa\;}\Omega.
\]
The initial distribution and the transient  matrix of the Erlang distribution are
\begin{equation}
 \phPi=(0,\ldots,0,1),
 \qquad
\phT=\kappa\begin{pmatrix}-1&&&\\1&-1&&\\&\ddots&\ddots&\\&&1&-1\end{pmatrix},
\label{eq:erlangexample}
\end{equation}
respectively. For $X\sim\Er_s(\kappa)$, the moments are given by
\begin{equation}
\E[X^j]=\frac{s(s+1)\cdots(s+j-1)}{\kappa^j},
\qquad j\geq1.
\label{eq:erlangmoments}
\end{equation}

\section{The RK--PH correspondence}
\label{sec:transform}

\subsection{Algebraic representation}
\label{sec:RKrepresentation}
We construct a PH representation from a Runge--Kutta tableau satisfying \((K)_r\). The stage conditions yield a transient subintensity matrix, while the output conditions provide nonnegative weights for the initial distribution.
\begin{proposition}
\label{prop:RKPH}
Let $(A,b)$ be an $s$-stage Runge--Kutta tableau of order at least one. Suppose $(K)_r$ holds for some $r>0$. Define
\begin{equation}
\begin{aligned}
 \Green&=I+rA, & \phT&=-\Green^{-1},\\
 \phW&=-rb^\trans\phT, & \sigma&=\phW\one.
\end{aligned}
\label{eq:transform}
\end{equation}
Then $\phT$ is a transient subintensity matrix. $\phW\geq\zero$ and $0<\sigma\leq1$ and \(\phPi=\phW/\sigma\)
is a probability row. We call \((\boldsymbol\pi,\mathbf T)\) the PH representation associated with \((A,b,r)\).
\end{proposition}

\begin{proof}
The first stage inequality in \eqref{eq:K} gives
\[
 I+\phT=rA(I+rA)^{-1}\geq0,
\]
so the off-diagonal entries of $\phT$ are nonnegative.
The second stage inequality gives
\[
 \phT\one=rA(I+rA)^{-1}\one-\one\leq\zero,
\]
so its row sums are nonpositive. Moreover,
$\phT=-(I+rA)^{-1}$ is nonsingular. According to Proposition~\ref{prop:transience}, $\phT$ is a valid transient subintensity matrix for a PH representation.

The output conditions in \eqref{eq:K} give $\phW\geq\zero$ and
$0\leq\sigma\leq1$. Consistency implies
$\phW\Green\one=rb^\trans\one=r>0$, so $\sigma>0$.
Thus $\phPi=\phW/\sigma$ is a probability row on the transient states.
\end{proof}

The transformation in \eqref{eq:transform} preserves the algebraic structure of the RK tableau and recasts the constraints in \((K)_r\) as admissibility conditions for the associated PH representation. Thus we can study SSP properties  without directly manipulating the numerous inequalities. Moreover,  the matrix $\phT$ and the vector $\bm w$  arise naturally from the Shu–Osher representation of the RK method. Collect the stages into \(U=(U_1,\ldots,U_s)^\top\) and write \(f(U)=(f(U_1),\ldots,f(U_s))^\top\). The RK equations can be rewritten as
\begin{equation}
\begin{aligned}
U&=\phExit u_n+(I+\phT)(U+(h/r)f(U)),\\
u_{n+1}&=(1-\sigma)u_n+\phW(U+(h/r)f(U)).
\end{aligned}
\end{equation}
Here \(I+\phT\) and \(\phExit\) give the convex-combination weights in the stage equations, while \(\phW\) and \(1-\sigma\) give the convex-combination weights in the output equation.

The same RK tableau also gives a PH distribution with initial absorbing mass, denoted by \(\PH(1-\sigma,\phW,\phT)\).  Conditioning on an initial transient state yields the normalized representation \(\mathrm{PH}(\phPi,\phT)\), where \(\phPi=\phW/\sigma\). Unless otherwise specified, we adopt the latter representation throughout this paper.

\subsection{Probabilistic interpretation of RK coefficients}
\label{sec:coefficients}
Uniformization represents a finite-state CTMC as a discrete-time
Markov chain driven by an independent Poisson clock. For a clock rate
$\nu\geq\max\lambda_i$,  the
transition probability matrix is $I+\Generator/\nu$ and 
\begin{equation}
\mathbf P(t)
=e^{t\Generator}
=e^{-\nu t}\sum_{n=0}^{\infty}
\frac{(\nu t)^n}{n!}
\left(I+\frac{\Generator}{\nu}\right)^n.
\label{eq:uniformization}
\end{equation}
For more details, see \cite[Section 1.3.4]{bladt2017matrix} and \cite{gross1984randomization, vandijk2018uniformization}.

Under the assumptions of Proposition~\ref{prop:RKPH},
$0<-t_{ii}\leq1$. We choose a clock of rate $\nu=1$ with transient
transition matrix
\[
 \mathbf P=I+\phT.
\]
At each clock event, a process in state $i$ moves to transient
state $j$ with probability $P_{ij}$ and is absorbed with probability
$1-(\mathbf P\one)_i$.

From the definitions of the transition probability matrix
and the Green matrix $\Green$ in \eqref{eq:greenmatrix}, we obtain
\(\Green=I+rA=I+\mathbf P\Green,\)
and hence
\[
 rA=\mathbf P\Green.
\]
Under the rate-one clock representation, the unit matrix $I$ accounts for the expected
occupation before the first clock event. Thus $rA$ is the expected
occupation matrix after the first clock event.

Since $U_{ij}=\E_i L_j$ and $c=A\one$, we have
\begin{align}
 ra_{ij}&=\E_i L_j-\delta_{ij},
 \label{eq:acooccupation}\\
 rc_i&=\E_i\tau-1,
 \label{eq:nodeoccupation}
\end{align}
For a chain starting from state $i$, $ra_{ij}$ is the expected time spent in state $j$ after the first clock event
while $rc_i$ is the expected time from that event to absorption.

With initial distribution $\phPi$, \eqref{eq:greenmatrix} and
Proposition~\ref{prop:RKPH} give
$\E_{\phPi}L_i=(\phPi\Green)_i=(r/\sigma)b_i$.
Since $b^\trans\one=1$, we have $\E_{\phPi}\tau=r/\sigma$ and
\begin{equation}
 b_i=\frac{\E_{\phPi}L_i}{\E_{\phPi}\tau}
     =\frac{\E_{\phPi}L_i}{\sum_{j=1}^s\E_{\phPi}L_j}.
\label{eq:weightoccupation}
\end{equation}
Thus $b_i$ is the normalized expected occupation time in state $i$.

\subsection{Moment identities and order conditions}
\label{sec:momentdictionary}
\begin{proposition}
\label{prop:rkmoments}
Under the assumptions of Proposition~\ref{prop:RKPH}, for $j\geq1$,
\begin{equation}
 \E\tau^j=j!\phPi\Green^j\one
 =\frac{j!r}{\sigma}b^\trans(I+rA)^{j-1}\one.
\label{eq:momentdictionary}
\end{equation}
If the tableau satisfies the chain order conditions
\eqref{eq:rkchainconditions} through order $p$, then
\begin{equation}
 \E\tau^j=\frac{j!r}{\sigma}
       \sum_{k=0}^{j-1}\binom{j-1}{k}\frac{r^k}{(k+1)!},
 \qquad 1\leq j\leq p.
\label{eq:generalrkmoments}
\end{equation}
\end{proposition}
\begin{proof}
The identities $\phPi=\phW/\sigma$ and
$\phW\Green=rb^\trans$ give
\[
 \E\tau^j=j!\phPi\Green^j\one
 =\frac{j!r}{\sigma}b^\trans\Green^{j-1}\one,
\]
which is \eqref{eq:momentdictionary}.
Expanding the matrix power and applying the chain conditions gives
\[
\begin{aligned}
 b^\trans(I+rA)^{j-1}\one
 &=\sum_{k=0}^{j-1}\binom{j-1}{k}r^k b^\trans A^k\one\\
 &=\sum_{k=0}^{j-1}\binom{j-1}{k}\frac{r^k}{(k+1)!},
 \qquad 1\leq j\leq p.
\end{aligned}
\]
Substitution yields \eqref{eq:generalrkmoments}.
\end{proof}

In particular, when the corresponding chain conditions hold,
\begin{equation}
 \E\tau=\frac r\sigma,\qquad
 \E\tau^2=\frac{r^2+2r}{\sigma},\qquad
 \E\tau^3=\frac{r^3+6r^2+6r}{\sigma}.
\label{eq:rkmomentsfour}
\end{equation}

This use of moment matching is analogous to the description of
optimal linear SSP polynomials in \cite{aithaddou2021computation}.
There the matching measure has bounded integer support and matches
Poisson moments; here the chain conditions prescribe moments of
a PH absorption time represented with $s$ transient states.

For $p\geq3$, the classical order conditions also include branching
conditions such as $b^\trans(c\cdot c)=1/3$. The full classical order conditions \eqref{eq:rkorderconditions} admit an
analogous occupation-reward expression. By \eqref{eq:weightoccupation},
they are equivalently
\begin{equation}
 \frac{\E_{\phPi}\displaystyle\int_0^\tau
             (F_{\mathfrak t})_{X_t}\,dt}
      {\E_{\phPi}\tau}
 =b^\trans F_{\mathfrak t}
 =\frac{1}{\gamma(\mathfrak t)},
 \qquad |\mathfrak t|\leq p.
\label{eq:treeoccupation}
\end{equation}
The first equality is an identity, with $(F_{\mathfrak t})_i$ as the
reward rate in state $i$; the second imposes the order condition.
This reformulates the standard tree conditions using state-dependent
rewards, rather than absorption-time moments alone.

\subsection{Three-stage examples}
\label{sec:rkexamples}
The following three-stage examples illustrate the PH representation and
the probabilistic interpretation of the RK coefficients. 
In both examples, the chain starts in state $3$.

\begin{example}[SSP(3,3) method]
\label{ex:ssp33PH}
Consider the classical three-stage, third-order explicit SSP method:
\[
\begin{aligned}
U_1&=u_n+h f(u_n),\\
U_2&=\frac34u_n+\frac14\bigl(U_1+h f(U_1)\bigr),\\
u_{n+1}&=\frac13u_n+\frac23\bigl(U_2+h f(U_2)\bigr).
\end{aligned}
\]
Its Butcher tableau is
\[
 A=
 \begin{pmatrix}
 0&0&0\\
 1&0&0\\
 \frac14&\frac14&0
 \end{pmatrix},
 \qquad
 b=\left(\frac16,\frac16,\frac23\right)^\trans,
 \qquad
 c=\left(0,1,\frac12\right)^\trans.
\]
At the radius of absolute monotonicity $r=1$, the transformation \eqref{eq:transform} gives
\[
 \Green=I+A=
 \begin{pmatrix}
 1&0&0\\
 1&1&0\\
 \frac14&\frac14&1
 \end{pmatrix},
 \qquad
 \phT=-\Green^{-1}=
 \begin{pmatrix}
 -1&0&0\\
 1&-1&0\\
 0&\frac14&-1
 \end{pmatrix},
\]
and
\[
 \phW=-b^\trans\phT=\left(0,0,\frac23\right),
 \qquad
 \sigma=\frac23,
 \qquad
 \phPi=\phW/\sigma=(0,0,1).
\]
The absorption paths are
\[
 \begin{array}{@{}c@{\;}l@{}}
 \overset{\scriptstyle 3/4}{
   \makebox[\widthof{$3\xrightarrow{\;1/4\;}2\xrightarrow{\;\;1\;\;}1\xrightarrow{\;\;1\;\;}4$}]{\kern0.25em
     \vrule width 0.4pt height 1.61ex depth 0.45ex
     \leaders\hrule height 1.61ex depth -1.52ex\hfill
     \kern-0.25em$\downarrow$}} & \\[0.5mm]
 3\xrightarrow{\;1/4\;}2\xrightarrow{\;\;1\;\;}1\xrightarrow{\;\;1\;\;}\Omega
 \end{array}
\]
where the labels on the arrows denote transition rates. State $3$ has total departure rate $3/4+1/4=1$. After an $\Exp(1)$ holding time, the process is absorbed directly with probability $3/4$ or jumps to state $2$ with probability $1/4$ and then follows $2\to1\to\Omega$ with independent $\Exp(1)$ holding times. The absorption time $\tau_1$ is therefore a mixture of Erlang distributions:
\[
 \tau_1 \sim \frac34\,\Er_1(1)+\frac14\,\Er_3(1).
\]

\end{example}

\begin{example}[Second-order implicit method]
\label{ex:midpointPH}
Consider the three-stage, second-order SDIRK method formed by three equal implicit-midpoint substeps:
\[
\begin{aligned}
U_1&=u_n+\frac h6f(U_1),\\
U_2&=U_1+\frac h6\left(f(U_1)+f(U_2)\right),\\
U_3&=U_2+\frac h6\left(f(U_2)+f(U_3)\right),\\
u_{n+1}&=U_3+\frac h6f(U_3).
\end{aligned}
\]
Its Butcher tableau is
\[
 A=
 \begin{pmatrix}
 \frac16&0&0\\[2pt]
 \frac13&\frac16&0\\[2pt]
 \frac13&\frac13&\frac16
 \end{pmatrix},
 \qquad
 b=\left(\frac13,\frac13,\frac13\right)^\trans,
 \qquad
 c=\left(\frac16,\frac12,\frac56\right)^\trans.
\]
At $r=6$, the transformation \eqref{eq:transform} gives
\[
 \Green=I+6A=2
 \begin{pmatrix}
 1&0&0\\
 1&1&0\\
 1&1&1
 \end{pmatrix},
 \qquad
 \phT=-\Green^{-1}=\frac12
 \begin{pmatrix}
 -1&0&0\\
 1&-1&0\\
 0&1&-1
 \end{pmatrix},
\]
and $\phW=\phPi=(0,0,1)$ with $\sigma=1$. The absorption path is
\[
 3\xrightarrow{\;1/2\;}
 2\xrightarrow{\;1/2\;}
 1\xrightarrow{\;1/2\;}
 \Omega.
\]
Each transient state is visited once with an independent $\Exp(1/2)$ holding time, so the absorption time $\tau_2$ satisfies
\[
 \tau_2\sim\Er_3(1/2).
\]

\end{example}

\section{Moment ratio bounds for phase-type distributions}
\label{sec:variance}
\subsection{Moment ratios and the minimum-variance bound}
\label{sec:momentratios}
For a nonnegative random variable $\tau$ with positive mean, suppose its moments are finite. Then we set $\mu_j=\E\tau^j$ and $\mu_0=1$.
The adjacent moment ratios are
\begin{equation}
 H_k(\tau)=\frac{\mu_{k-1}\mu_{k+1}}{\mu_k^2},
\label{eq:momentratios}
\end{equation}
where $k$ is a positive integer.
In particular,
\begin{equation}
 H_1(\tau)=\frac{\mu_2}{\mu_1^2}=1+\frac{\Var(\tau)}{(\E\tau)^2}, \qquad H_2(\tau)=\frac{\mu_1\mu_3}{\mu_2^2}.
\label{eq:firsttworatios}
\end{equation}
Thus $H_1$ measures the variance relative to the squared mean,
while $H_2$ compares the first three moments.
Cauchy--Schwarz gives $H_k\geq1$ for any nonnegative variable with finite moments.
For PH distributions, the number of transient states may yield stronger lower bounds on the moment ratios. This is exactly what we are concerned with in this section.

For the Erlang distribution, the moment formula
\eqref{eq:erlangmoments} gives
\begin{equation}
 H_k(\Er_s(\kappa))=\frac{s+k}{s+k-1},\qquad k\geq1.
\label{eq:erlangratios}
\end{equation}
This value is independent of the rate $\kappa$. Hence, a natural conjecture about moment ratio bounds is that 
\begin{equation}
 H_k(\tau)\ge\frac{s+k}{s+k-1},\qquad k\geq1.
\label{eq:ratiosbounds}
\end{equation}
For $k=1$, it is  proved by the following classical minimum-variance theorem.

\begin{theorem}[Aldous--Shepp {\cite{aldous1987least}}]
\label{thm:aldousfixed}
Let $(X_t)_{t\geq0}$ be a continuous-time Markov chain with
$s$ transient states and one absorbing state, started from a
fixed transient state. Let $\tau$ be its absorption time. Then
\begin{equation}
 \frac{\Var(\tau)}{(\E\tau)^2}\geq\frac1s.
\label{eq:variancebound}
\end{equation}
Equality holds if and only if the chain is an Erlang process.
\end{theorem}

They proved this theorem by constructing a martingale that expresses
the absorption-time variance as the expected sum of squared changes
in the mean remaining absorption time along a sample path.
 This construction also identifies the Erlang process through the equality conditions. 
 For further details, see
\cite{aldous1987least}.

The theorem also gives the bound for a random initial 
state. Applying Theorem~\ref{thm:aldousfixed}
conditional on the initial state $X_0$ and using the law of total variance gives
\[
\begin{aligned}
 \Var(\tau)
 &\geq\frac1s\E\!\left[\bigl(\E[\tau\mid X_0]\bigr)^2\right]
   +\Var\!\left(\E[\tau\mid X_0]\right)\\
 &=\frac{(\E\tau)^2}{s}
   +\left(1+\frac1s\right)\Var\!\left(\E[\tau\mid X_0]\right)\\
 &\geq\frac{(\E\tau)^2}{s}.
\end{aligned}
\]
Consequently, $H_1(\tau)\geq(s+1)/s$ holds for any initial transient distribution.
Equality holds if and only if the absorption time has an Erlang distribution.

\subsection{Third-moment ratio bound}

The corresponding bound for $H_2$ is given by the following theorem.
\begin{theorem}[Third-moment bound]
\label{thm:thirdmoment}
If $\tau\sim\PH(\phPi,\phT)$ has a representation with $s$ transient states, then
\begin{equation}
 H_2(\tau)\geq\frac{s+2}{s+1}.
\label{eq:thirdmomentbound}
\end{equation}
Equality holds if and only if $\tau$ has an  
Erlang distribution with $s$ phases.
\end{theorem}

Osogami and Harchol-Balter \cite{osogami2006closed}
proved this bound for acyclic PH distributions by
induction on the number of phases.
We retain their framework to treat general PH representations. 

Firstly, it suffices to prove the bound for a fixed initial state. Then the absorption time is decomposed as $\tau=L_1+Z$, where $L_1$ is the total time spent in the initial state and $Z$ is the remaining occupation time. A Laplace transform shows that $Z$ admits an $(s-1)$-state PH representation, possibly with initial absorbing mass. Finally, we calculate the mixed moments of $(L_1,Z)$ from the joint Laplace transform and finish the induction by an algebraic decomposition.

For an $s$-state phase-type random variable $\tau$, define the difference
\begin{equation}
 \Delta_s(\tau)=\mu_1\mu_3-\frac{s+2}{s+1}\mu_2^2.
\label{eq:targetdefect}
\end{equation}
The desired inequality is exactly $\Delta_s(\tau)\geq0$.
For $s=1$, the absorption time $\tau=L_1$ is exponential and $\Delta_1(L_1)=0$. This is the base case of the induction.

\paragraph{Fix the initial state}
\label{sec:fixedinitial}
\label{sec:thirdmomentstatement}

The absorption-time law under initial distribution $\phPi$
is a mixture of the laws corresponding to fixed initial states.
Suppose that each such law satisfies the bound
\[
 \E_i\tau^3
 \geq c\frac{(\E_i\tau^2)^2}{\E_i\tau},
 \qquad i=1,\ldots,s,
\]
for a common constant $c>0$. The 
weighted Cauchy--Schwarz inequality gives
\begin{equation}
   \E_{\phPi}\tau^3
 =\sum_i\pi_i\E_i\tau^3
 \geq c\sum_i\pi_i
 \frac{(\E_i\tau^2)^2}{\E_i\tau}
 \geq c\frac{(\E_{\phPi}\tau^2)^2}{\E_{\phPi}\tau}.
 \label{eq:mixturebound}
\end{equation}
Thus the same bound holds for any initial transient distribution.
Equality holds if and only if every conditional law with $\pi_i>0$
attains the bound and the ratios
$\E_i\tau^2/\E_i\tau$
are equal over these initial states.
Hence, it is sufficient to consider a fixed initial state. We label this state as state 1. In the following discussion, probabilities and expectations are taken with $X_0=1$, unless a different initial state or distribution is explicitly specified.

\paragraph{Elimination of the initial state}
\label{sec:stateelimination}
Write the absorption time as
\begin{equation}
 \tau=L+Z,
\label{eq:timepartition}
\end{equation}
where $L=L_1$ is the total time spent in state $1$ and $Z=\sum_{i=2}^s L_i$ is the remaining occupation time. 
In order to determine the distributions of  $Z$ and $L$, we
partition the transient subintensity matrix as
\begin{equation}
 \phT=\begin{pmatrix}-\lambda_1&\mathbf r^\trans\\
 \mathbf q&\mathbf B\end{pmatrix}.\qquad
 \phExit=\begin{pmatrix}t_1\\\mathbf t_B\end{pmatrix},
\label{eq:Sblock}
\end{equation}
The row $\mathbf r^\trans$ contains rates from state 1 into the other
states, and $\mathbf q$ contains their return rates to state 1.
 The block $\mathbf B$ describes the remaining process stopped upon either returning to state 1 or being absorbed. It is transient and $-\mathbf B\one=\mathbf q+\mathbf t_B$.

\begin{lemma}
\label{lem:smallerPH}
The variable $Z$ has a PH representation with $s-1$ transient states,
possibly with initial absorbing mass. Set
$\widehat{\phT}=\mathbf B+{\mathbf q\mathbf r^\trans}/{\lambda_1}$.
Then
\begin{equation}
 Z\sim\PH\Bigl(
  \frac{t_1}{\lambda_1},
  \frac{\mathbf r^\trans}{\lambda_1},
\widehat{\phT}\Bigr).
\label{eq:censored}
\end{equation}
Moreover, the variable $L$ has an exponential distribution. Set $\lambda_*=\lambda_1-\mathbf r^\trans(-\mathbf B)^{-1}\mathbf q$. Then
\begin{equation}
    L\sim\Exp(\lambda_*).
\end{equation}
\end{lemma}
\begin{proof}
Define the joint Laplace transform of \((L,Z)\) by
\[
 \Phi_i(v,u)=\E[e^{-vL-uZ}\mid X_0=i],\quad v,u\geq0.
\]
Write $\boldsymbol\Phi_B(v,u)=(\Phi_2(v,u),\ldots,\Phi_s(v,u))^\trans$
for the column vector of transforms corresponding to initial states
$2,\ldots,s$.
In particular, $\Phi_1(0,u)=\E e^{-uZ}$ is the transform required for this lemma.

Write $\lambda_i=-t_{ii}$, where $t_{ij}$ are the entries of $\phT$.
Let $H$ be the first holding time at state $1$ and $I$
the state reached at the first jump. Then $H\sim\Exp(\lambda_1)$ and $H$
is independent of $I$. All of $Z$ is accumulated
after the first jump. 
Conditioning on $H$, if $I$ is transient, the subsequent occupation times have joint
transform $\Phi_i(v,u)$ for $i=2,\ldots,s$.  If $I$ is absorbing, both
subsequent occupation times are zero. Consequently,
\[
\begin{aligned}
 \Phi_1(v,u)
 &=\E\!\left[e^{-vH}e^{-v(L-H)-uZ}\right]\\
 &=\E[e^{-vH}]
 \left(\Prob(I=\Omega)
 +\sum_{i=2}^s\Prob(I=i)\Phi_i(v,u)
   \right)\\
 &=\frac{\lambda_1}{\lambda_1+v}
   \left(
     \frac{t_1}{\lambda_1}
     +\frac{\mathbf r^\trans\boldsymbol\Phi_B(v,u)}{\lambda_1}
   \right).
\end{aligned}
\]
The first holding time contributes
to $Z$ if the chain starts from $i\neq1$. 
The next jump leads to state 1, to another remaining state, or to
absorption. The same conditioning gives
\[
 \Phi_i(v,u)=\frac{\lambda_i}{\lambda_i+u}
 \left(\frac{t_{i1}}{\lambda_i}\Phi_1(v,u)
       +\sum_{\substack{j=2\\j\ne i}}^s
          \frac{t_{ij}}{\lambda_i}\Phi_j(v,u)
       +\frac{t_i}{\lambda_i}\right).
\]
Collecting these equations in matrix form, we obtain
\begin{equation}
 \begin{aligned}
 (v+\lambda_1)\Phi_1&=t_1+\mathbf r^\trans\boldsymbol\Phi_B,\\
 (uI-\mathbf B)\boldsymbol\Phi_B&=\mathbf q\Phi_1+\mathbf t_B.
 \end{aligned}
\label{eq:firstjumpjoint}
\end{equation}

Set $v=0$ to obtain the transform of
$Z$ alone. The equation
then gives
\[
\begin{aligned}
 \Phi_1(0,u)
 &=\frac{t_1}{\lambda_1}
   +\frac{\mathbf r^\trans}{\lambda_1}\boldsymbol\Phi_B(0,u),\\
 \left(uI-\widehat{\phT}\right)\boldsymbol\Phi_B(0,u)
 &=\mathbf t_B+\frac{t_1}{\lambda_1}\mathbf q.
\end{aligned}
\]
The off-diagonal entries of $\widehat{\phT}$ are nonnegative. 
Set
\(\widehat{\phExit}=\mathbf t_B+\frac{t_1}{\lambda_1}\mathbf q\).
The row sums of $\widehat{\phT}$ satisfy
\[
 -\widehat{\phT}\one
 =\mathbf q+\mathbf t_B
   -\frac{\mathbf q\mathbf r^\trans\one}{\lambda_1}
 =\widehat{\phExit}\geq\zero.
\]
Using the block form \eqref{eq:Sblock}, we obtain 
\[
\det(\widehat{\phT})=-\frac1{\lambda_1}\det(\phT)\neq0.
\]
Thus $\widehat{\phT}$ is a valid
transient subintensity matrix according to Proposition~\ref{prop:transience}. Also, $uI-\widehat{\phT}$ is invertible
for every $u\geq0$.
The proposed initial row and absorbing mass are nonnegative and
have total mass
\[
 \frac{\mathbf r^\trans\one}{\lambda_1}
 +\frac{t_1}{\lambda_1}=1.
\]

Solving for $\boldsymbol\Phi_B(0,u)$ and substituting into
$\Phi_1(0,u)$ gives
\begin{equation}
 \E e^{-uZ}=\Phi_1(0,u)=\frac{t_1}{\lambda_1}
 +\frac{\mathbf r^\trans}{\lambda_1}
 (uI-\widehat{\phT})^{-1}\widehat{\phExit}.
\label{eq:smallerPHtransform}
\end{equation}
The right-hand side is the Laplace transform of the PH
representation including its mass at zero.
By uniqueness of Laplace transforms, $Z$ has the stated PH representation.

Since $\E L=U_{11}$ is the first entry of $\Green$,
the block inverse formula gives
$\lambda_*=1/\E L>0$.
Setting $u=0$ in \eqref{eq:firstjumpjoint}  yields
\[
 \E e^{-vL}=\Phi_1(v,0)
 =\frac{\lambda_*}{v+\lambda_*}.
\]
Thus $L$ has the stated exponential distribution.
\end{proof}

Kulkarni \cite{kulkarni1989} and Bladt--Nielsen
\cite[Theorem~3.1.33]{bladt2017matrix} use a jump-chain construction
to show that accumulated nonnegative state-dependent rewards
have a PH distribution.
Lemma~\ref{lem:smallerPH} is a special case of their theorem. We keep this proof since the joint Laplace transform
derived will also be used to compute the mixed moments of $L$ and $Z$.

\paragraph{Mixed moments}

Since the moment ratio $H_2$ is invariant under time scaling,
we may multiply all occupation times by $\lambda_*$ and replace
$\phT$ by $\phT/\lambda_*$. By Lemma~\ref{lem:smallerPH}, the rescaled occupation time satisfies
\[
 L\sim\Exp(1),\qquad \E L=1.
\]
The same notation is retained for the rescaled variables and generator.

Write $z_j=\E Z^j$ for $j\ge1$. Expanding $\tau=L+Z$ gives
\[
\begin{aligned}
 \mu_1&=1+z_1,\\
 \mu_2&=2+2\E(LZ)+z_2,\\
 \mu_3&=6+3\E(L^2Z)+3\E(LZ^2)+z_3.
\end{aligned}
\]
We now compute the mixed moments $\E(LZ)$, $\E(L^2Z)$, and
$\E(LZ^2)$ with the joint transform
$\Phi_1(v,u)=\E e^{-vL-uZ}$.

Define the auxiliary functions 
  \[
 F(u)=\E e^{-uZ},\qquad
 d(u)=\lambda_1-\mathbf r^\trans(uI-\mathbf B)^{-1}\mathbf q.
\]
Eliminating $\boldsymbol\Phi_B$ from \eqref{eq:firstjumpjoint} gives
\[
 (v+d(u))\Phi_1(v,u)
 =t_1+\mathbf r^\trans(uI-\mathbf B)^{-1}\mathbf t_B.
\]
The right-hand side does not depend on $v$. Setting $v=0$ and
using $\Phi_1(0,u)=F(u)$ identifies it as $d(u)F(u)$. Hence
\[
 \Phi_1(v,u)=F(u)\frac{d(u)}{v+d(u)}.
\]
Differentiating $\Phi_1(v,u)$ once and twice with respect to $v$ and setting $v=0$ yields
\begin{equation}
\begin{aligned}
 \E(Le^{-uZ})=\frac{F(u)}{d(u)},\qquad
 \E(L^2e^{-uZ})=\frac{2F(u)}{d(u)^2}.
\end{aligned}
\label{eq:jointtransform}
\end{equation}

Define the first two derivatives of $d$ at zero as 
\begin{equation}
\begin{aligned}
 \eta_1=d'(0)=\mathbf r^\trans(-\mathbf B)^{-2}\mathbf q,\qquad
 \eta_2=-d''(0)=2\mathbf r^\trans(-\mathbf B)^{-3}\mathbf q.
\end{aligned}
\label{eq:returnmoments}
\end{equation}
Since $\mathbf B$ is a transient subintensity matrix,
$(-\mathbf B)^{-1}$ is entrywise nonnegative.
Together with $\mathbf r,\mathbf q\geq\zero$,
the definition implies $\eta_1,\eta_2$ are nonnegative.

We differentiate \eqref{eq:jointtransform} with respect to \(u\) and set \(u=0\).
Using $d(0)=1$ and
$F^{(j)}(0)=(-1)^jz_j$, we obtain
\begin{equation}
\begin{aligned}
 \E(LZ)&=-\left.\frac{d}{du}\frac{F(u)}{d(u)}\right|_{u=0}
          =z_1+\eta_1,\\
 \E(L^2Z)&=-\left.\frac{d}{du}\frac{2F(u)}{d(u)^2}\right|_{u=0}
          =2z_1+4\eta_1,\\
 \E(LZ^2)&=\left.\frac{d^2}{du^2}\frac{F(u)}{d(u)}\right|_{u=0}
          =z_2+2z_1\eta_1+2\eta_1^2+\eta_2.
\end{aligned}
\label{eq:mixedoccupationmoments}
\end{equation}
In particular, $\E L=1$ implies
$\eta_1=\E(LZ)-(\E L)(\E Z)=\operatorname{Cov}(L,Z)$.

Substituting \eqref{eq:mixedoccupationmoments} into the expansions
of $\tau=L+Z$  yields
\begin{equation}
\begin{aligned}
 \mu_1&=1+z_1,\\
 \mu_2&=2+2z_1+z_2+2\eta_1,\\
 \mu_3&=6+6z_1+3z_2+z_3
              +6(z_1+2)\eta_1+6\eta_1^2+3\eta_2.
\end{aligned}
\label{eq:threemomentexpansion}
\end{equation}

\paragraph{Nonnegative decomposition}
If $z_1=0$, then $Z=0$ almost surely and $\Delta_s(\tau)=\Delta_s(L)=2(s-1)/(s+1)>0$. For the remainder of the proof, assume that $z_1>0$.

The variable $Z$ may have  absorbing mass at zero, while the
induction hypothesis concerns PH distributions without mass at zero.
Let $p=\Prob(Z>0)$ and let $Z^*$ have the
conditional law of $Z$ given $Z>0$.
Since $z_j=p\E\bigl[(Z^*)^j\bigr]$ for $j\geq1$, we have
$$H_2(Z)=H_2(Z^*),\qquad H_1(Z)=H_1(Z^*)/p.$$
The induction hypothesis and the variance bound yield that
\begin{equation}
 z_1z_3\geq\frac{s+1}{s}z_2^2,\qquad
 z_2\geq\frac{s}{s-1}z_1^2.
\label{eq:smallmomentbounds}
\end{equation}

In $\Delta_s(\tau)$, the coefficients of $z_3$ and $\eta_2$ are
$1+z_1$ and $3(1+z_1)$, respectively. To use the induction
hypothesis, write
\[
 (1+z_1)z_3
 =\frac{1+z_1}{z_1}\Delta_{s-1}(Z)
  +\frac{(1+z_1)(s+1)}{s}\frac{z_2^2}{z_1}.
\]
Substitution into \eqref{eq:targetdefect} therefore gives
\begin{equation}
 \Delta_s(\tau)=\frac{1+z_1}{z_1}\Delta_{s-1}(Z)
               +3(1+z_1)\eta_2+\mathcal R_s,
\label{eq:exactdefect}
\end{equation}
where
\begin{equation}
\begin{aligned}
 \mathcal R_s={}&(1+z_1)\left[
 6+6z_1+3z_2+\frac{s+1}{s}\frac{z_2^2}{z_1}
 +6(z_1+2)\eta_1+6\eta_1^2\right]\\
 &-\frac{s+2}{s+1}(2+2z_1+z_2+2\eta_1)^2.
\end{aligned}
\label{eq:remainder}
\end{equation}
The first two terms in \eqref{eq:exactdefect} are nonnegative.
The remainder depends only on $z_1,z_2,\eta_1$. Set
\begin{equation}
 \varepsilon_m=\frac{z_1}{s-1}-1,\qquad
 \varepsilon_v=\frac{z_2}{z_1}-\frac{s}{s-1}z_1.
\label{eq:erlangdefects}
\end{equation}
Direct expansion gives
\begin{equation}
\begin{aligned}
 (s+1)\mathcal R_s={}&2(s-1)(\eta_1-\varepsilon_m)^2+\frac{z_1^2}{s}(\varepsilon_v+s\varepsilon_m)^2
   +\frac{z_1}{s}\varepsilon_v^2\\
 &+z_1(s+2)(\varepsilon_v-2\eta_1+\varepsilon_m)^2\\
 &+2(s-1)z_1\eta_1^2+z_1(3s-2)\varepsilon_m^2\\
 &+\frac{z_1}{s-1}\left[(s^2+1)z_1+(s-1)^2\right]\varepsilon_v\\
 &+2z_1\left[(s-1)z_1+3s-1\right]\eta_1.
\end{aligned}
\label{eq:gsos}
\end{equation}
Since
$\varepsilon_v,\eta_1\geq0$, this is a sum of six squares with
positive coefficients and two nonnegative terms. Thus $\mathcal R_s\geq0$ and
$\Delta_s(\tau)\geq0$. 
This completes the induction.

\paragraph{Equality characterization}
First fix the initial state and use the normalization $\E L=1$.
For $s=1$, the law is exponential. For $s\geq2$, equality excludes
$z_1=0$. The strictly positive coefficients of
$\varepsilon_v^2$, $\eta_1^2$, and $\varepsilon_m^2$ in \eqref{eq:exactdefect} and
\eqref{eq:gsos} force
\[
 z_1=s-1,\qquad z_2=s(s-1),\qquad \eta_1=\eta_2=0.
\]
Equation~\eqref{eq:threemomentexpansion}
then gives
\[
 \E\tau=s,\qquad \Var(\tau)=s.
\]
The normalized PH time attains the minimum-variance bound. The equality characterization shows that the representing chain is a unit-rate Erlang process.
Time scaling changes only the rate and preserves the Erlang form of the distribution.

For a random initial state, equality in the mixture argument in \eqref{eq:mixturebound} requires each conditional law with $\pi_i>0$ to attain the bound. Hence, each conditional law is an Erlang law with $s$ phases. These laws must also have the same ratio $\E_i\tau^2/\E_i\tau$. For an Erlang law of rate $\kappa$, this ratio is $(s+1)/\kappa$. Hence their rates coincide, and their mixture is the same Erlang law.

Conversely,  \eqref{eq:erlangratios} gives $H_2=(s+2)/(s+1)$.

\subsection{Moment ratio bounds for acyclic PH distributions}
\label{sec:acyclicratios}
Sections~\ref{sec:momentratios} and~\ref{sec:thirdmomentstatement}
establish \eqref{eq:ratiosbounds} for $k=1,2$.
To investigate the higher-order cases, we tested \(k=3,4,5\) for \(s=2,\ldots,12,16,20\) by randomly sampling cyclic PH representations and performing numerical minimization.
No violations of the bound were found within a tolerance of $10^{-10}$.
However, a general proof for $k\geq3$ remains open.

For acyclic PH distributions, the following theorem establishes the
bound for every $k\geq1$ and characterizes equality.

\begin{theorem}[Acyclic moment ratio bound]
\label{thm:acyclicratios}
If $\tau\sim\PH(\phPi,\phT)$ has an acyclic representation with
$s$ transient states, then
\begin{equation}
 H_k(\tau)=\frac{\mu_{k-1}\mu_{k+1}}{\mu_k^2}
 \geq\frac{s+k}{s+k-1},
 \qquad k\geq1.
\label{eq:acyclicmomentbound}
\end{equation}
For each fixed $k\geq1$, equality holds if and only if $\tau$ has an
  Erlang distribution with $s$ phases.
\end{theorem}

\begin{proof}
Since the representation is acyclic,
we may relabel the transient states so that $\phT$ is lower triangular.
Every jump between transient states decreases the state index.
Hence, each transient state is visited at most once.

Fix a path of positive probability,
\[
  J=(i_1,\ldots,i_\ell,\Omega),
 \qquad s\geq i_1>\cdots>i_\ell \geq 1,\qquad 1\leq\ell\leq s.
\]
Let $\tau_J$ be the absorption
time conditional on this path, and set
$\mu_j^{(J)}=\E\tau_J^j$ and $\mu_0^{(J)}=1$.
At state $i$, the holding time is exponential with rate
$\lambda_i=-t_{ii}$ and is independent of the next state.
Conditional on $J$, the successive holding times are independent,
so
\[
 \tau_J=\sum_{r=1}^{\ell}\frac{1}{\lambda_{i_r}}E_r,
 \qquad E_r\overset{\mathrm{iid}}{\sim}\Exp(1).
\]
Set $\theta_r=\lambda_{i_r}^{-1}$ for $1\leq r\leq\ell$ and
$\theta_r=0$ for $\ell<r\leq s$. We can  write
\begin{equation}
\tau_J=\sum_{r=1}^s\theta_rE_r.
\end{equation}

To calculate the moments, define the complete homogeneous symmetric
polynomials
\[
 h_j(\theta)=
 \sum_{\substack{a_1+\cdots+a_s=j\\a_1,\ldots,a_s\geq0}}
 \theta_1^{a_1}\cdots\theta_s^{a_s},
 \qquad h_0(\theta)=1.
\]
The multinomial
formula and $\E[E_r^j]=j!$ give
\begin{equation}
 \mu_j^{(J)}
 =\sum_{a_1+\cdots+a_s=j}
   \frac{j!}{a_1!\cdots a_s!}
   \prod_{r=1}^s\theta_r^{a_r}\E[E_r^{a_r}]
 =j!h_j(\theta).
\label{eq:pathhomogeneousmoments}
\end{equation}

For nonnegative $\theta_1,\ldots,\theta_s$, Schur's inequality
\cite{pecaric2009integral} states that the normalized polynomials
\[
 a_j(\theta)=\frac{h_j(\theta)}{\binom{s+j-1}{j}}
\]
satisfy
\begin{equation}
 a_k(\theta)^2\leq a_{k-1}(\theta)a_{k+1}(\theta),
 \qquad k\geq1.
\label{eq:schurcomplete}
\end{equation}
 Since $\mu_j^{(J)}={(s+j-1)!}/{(s-1)!}a_j(\theta)$, \eqref{eq:schurcomplete} gives the path bound
 \begin{equation}\frac{\mu_{k-1}^{(J)}\mu_{k+1}^{(J)}}{(\mu_k^{(J)})^2}=\frac{(s+k-2)!\,(s+k)!}{\bigl((s+k-1)!\bigr)^2}\, \frac{a_{k-1}(\theta)a_{k+1}(\theta)}{a_k(\theta)^2}\geq\frac{s+k}{s+k-1}.\label{eq:pathmomentbound}\end{equation}

For an arbitrary initial transient distribution, the absorption-time law is a mixture of conditional laws of $\tau_J$.
Let $p_J$ be the probability of path $J$, so that
$\mu_j=\sum_Jp_J\mu_j^{(J)}$.
Set $c_{s,k}=(s+k)/(s+k-1)$. Applying \eqref{eq:pathmomentbound}
and the weighted Cauchy--Schwarz inequality gives
\begin{equation}
 \mu_k
 \leq\frac{1}{\sqrt{c_{s,k}}}
       \sum_Jp_J\sqrt{\mu_{k-1}^{(J)}\mu_{k+1}^{(J)}}
 \leq\frac{\sqrt{\mu_{k-1}\mu_{k+1}}}{\sqrt{c_{s,k}}}.
\label{eq:acyclicpathmixture}
\end{equation}
This proves \eqref{eq:acyclicmomentbound}.

We now characterize equality for a fixed \(k\geq1\), letting \(J\)
range over all paths with probability \(p_J>0\).
Suppose that equality holds in \eqref{eq:acyclicmomentbound}.
Then both inequalities in \eqref{eq:acyclicpathmixture} must be
equalities. Since the first inequality is obtained by summing
the path bounds with positive weights, every path \(J\) of
positive probability must satisfy \(H_k(\tau_J)=c_{s,k}\).
If such a path visits only \(\ell<s\) transient states, applying
the same path bound with \(\ell\) variables gives
\[
 H_k(\tau_J)
 \geq\frac{\ell+k}{\ell+k-1}
 >\frac{s+k}{s+k-1}
 =c_{s,k},
\]
a contradiction. Hence every path of positive probability visits
all \(s\) transient states. Since the state indices decrease
strictly along each path, the only possible such path is
\[
 J_*=(s,s-1,\ldots,1,\Omega).
\]
Consequently, \(p_{J_*}=1\) and \(\pi_s=1\), so the absorption-time
law coincides with the conditional law along \(J_*\).
In particular, the Cauchy--Schwarz step in
\eqref{eq:acyclicpathmixture} is automatically an equality.
Equality in the path bound is therefore equivalent to equality
in Schur's inequality \eqref{eq:schurcomplete}, which holds if
and only if \(\theta_1=\cdots=\theta_s\).
Thus all holding rates are equal to some \(\lambda>0\), and
\(\tau\) is the sum of \(s\) independent exponential random
variables with rate \(\lambda\), hence has an Erlang distribution
with \(s\) phases.

Conversely, if \(\tau\) has an Erlang distribution with \(s\) phases,
then \eqref{eq:erlangratios} gives \(H_k(\tau)=(s+k)/(s+k-1)\) for
every \(k\geq1\).
\end{proof}

\section{Optimal radii of absolute monotonicity}
\label{sec:applications}
For a Runge--Kutta tableau $(A,b)$ satisfying $(K)_r$, the associated RK--PH
representation and the moment identities \eqref{eq:rkmomentsfour} give
\begin{equation}
 H_1(\tau)=\sigma\left(1+\frac2r\right),\qquad
 H_2(\tau)=\frac{r^2+6r+6}{(r+2)^2},
\label{eq:Hrk}
\end{equation}
for methods of orders at least two and three, respectively.
Combined with the moment bounds in Section~\ref{sec:variance},
these identities establish the conjectured radii of absolute monotonicity. The equality conditions also identify the
attaining tableaux uniquely, up to a simultaneous permutation of stages.
We write
\begin{equation}
 R_{s,p}=\sup_{(A,b)}R(A,b),
\label{eq:optimalRKradius}
\end{equation}
where the supremum is taken over $s$-stage Runge--Kutta tableaux
of order at least $p$.

\subsection{Second-order methods}
\label{sec:order2optimal}
The conjectured optimal radius of absolute monotonicity at order two is $2s$
\cite{ketcheson2009optimal}, attained by the equal-step composition
of the implicit midpoint rule. The following theorem gives the bound and uniqueness for general implicit tableaux.

\begin{theorem}
\label{thm:order2}
For every $s\geq1$, the optimal radius of absolute monotonicity is
\[
 R_{s,2}=2s.
\]
Up to a simultaneous permutation of stages, the unique attaining
tableau is the $s$-fold equal-step composition of the implicit midpoint
rule,
\begin{equation}
 A=\frac1{2s}\begin{pmatrix}
  1&&&\\
  2&1&&\\
  \vdots&\vdots&\ddots&\\
  2&2&\cdots&1\\
  2&2&\cdots&2&1
 \end{pmatrix},\qquad
 b=\frac1{s}\begin{pmatrix}
  1\\1\\\vdots\\1
 \end{pmatrix}.
\label{eq:midpointtableau}
\end{equation}
\end{theorem}

\begin{proof}
Let \((A,b)\) be the Butcher tableau of an \(s\)-stage SSP Runge–Kutta method of order at least two, and let \(r>0\) be feasible. The minimum-variance bound in Section~\ref{sec:momentratios}, together with \eqref{eq:Hrk} and \(\sigma\leq1\), gives
\begin{equation}
  1+\frac1s
\leq \sigma\left(1+\frac2r\right)
\leq 1+\frac2r.
\label{eq:bound-two}
\end{equation}
Hence \(r\leq2s\), and therefore \(R(A,b)\leq2s\). Since the tableau \eqref{eq:midpointtableau} has radius of absolute monotonicity \(2s\), this bound is sharp and \(R_{s,2}=2s\).

It remains to prove uniqueness. Suppose that \(R(A,b)=2s\). By the definition of \(R(A,b)\), there exists a sequence of feasible parameters \(r_n\to2s\).
The conditions $(K)_{r_n}$ imply that \(r_nA(I+r_nA)^{-1}\) is nonnegative with row sums at most one. Thus
\[
(I+r_nA)^{-1}=I-r_nA(I+r_nA)^{-1}
\]
is bounded. Passing to a convergent subsequence shows that
\(I+2sA\) is invertible, so \(r=2s\) is feasible.

At $r=2s$, let \((\phPi,\phT)\) be the PH
representation of \((A,b,2s)\) and $\tau\sim\PH(\phPi,\phT)$. The equality in the estimate \eqref{eq:bound-two} forces $\sigma=1$ and $H_1(\tau)=(s+1)/s$. Thus the minimum-variance bound theorem ensures that the representing chain is an
Erlang process with \(s\) phases.
 Since 
$\E\tau=r/\sigma=2s,$  the common rate is $\kappa=1/2$.
The Erlang representation gives
\[
\phT=\frac12
\begin{pmatrix}
-1&&&\\
1&-1&&\\
&\ddots&\ddots&\\
&&1&-1
\end{pmatrix},\quad
\Green
=2\begin{pmatrix}
1&&&\\
1&1&&\\
\vdots&\vdots&\ddots&\\
1&1&\cdots&1
\end{pmatrix},\quad
\phPi=(0,\cdots,0,1).
\]
The transform \eqref{eq:transform} gives
\[
A=\frac{\Green-I}{2s},
\qquad
b=\frac{1}{2s}(\phPi\Green)^\trans
=\frac1s\one.
\]
These are the coefficients in \eqref{eq:midpointtableau}.
Hence, any tableau attaining \(R(A,b)=2s\) coincides with
\eqref{eq:midpointtableau} after a simultaneous permutation
of stages. 
\end{proof}

\subsection{Third-order methods}
\label{sec:order3optimal}
At order three, the conjectured optimal radius of absolute monotonicity is
$s-1+\sqrt{s^2-1}$ \cite{ketcheson2009optimal}.
The SDIRK family of Ferracina and Spijker \cite{ferracina2008strong}
attains this value. We now prove its optimality and uniqueness
among general implicit tableaux.

\begin{theorem}
\label{thm:order3}
For every $s\geq2$, the optimal radius of absolute monotonicity is
\[
 R_{s,3}=\rho_s:=s-1+\sqrt{s^2-1}.
\]
Up to a simultaneous permutation of stages, the unique attaining
tableau is
\begin{equation}
 A=\begin{pmatrix}
  \gamma&&&&\\[1mm]
  \delta&\gamma&&&\\[1mm]
  \vdots&\vdots&\ddots&&\\[1mm]
  \delta&\delta&\cdots&\gamma&\\[1mm]
  \delta&\delta&\cdots&\delta&\gamma
 \end{pmatrix},\qquad
 b=\frac1s\begin{pmatrix}
  1\\[1mm]1\\[1mm]\vdots\\[1mm]1
 \end{pmatrix},
\label{eq:thirdtableau}
\end{equation}
where
\begin{equation}
 \gamma=\frac12\left(1-\sqrt{\frac{s-1}{s+1}}\right),\qquad
 \delta=\frac1{\sqrt{s^2-1}}.
\label{eq:thirdparams}
\end{equation}
\end{theorem}

\begin{proof}
Let \((A,b)\) be the Butcher tableau of an \(s\)-stage SSP
Runge--Kutta method of order at least three, and let \(r>0\)
be feasible. Theorem~\ref{thm:thirdmoment} and \eqref{eq:Hrk} give
\begin{equation}
 \frac{s+2}{s+1}\leq\frac{r^2+6r+6}{(r+2)^2}.
\label{eq:bound-three}
\end{equation}
Equivalently, \(r^2-2(s-1)r-2(s-1)\leq0\), so
\(R(A,b)\leq\rho_s\). Since the third-order tableau
\eqref{eq:thirdtableau} has radius of absolute monotonicity
\(\rho_s\), this bound is sharp and
\(R_{s,3}=\rho_s\).

Suppose that \(R(A,b)=\rho_s\).
The argument in the proof of Theorem~\ref{thm:order2} shows that
\(r=\rho_s\) is feasible. 

Let \((\phPi,\phT)\) be the PH
representation of \((A,b,\rho_s)\), and let
\(\tau\sim\PH(\phPi,\phT)\). At \(r=\rho_s\),
\eqref{eq:bound-three} gives \(H_2(\tau)=(s+2)/(s+1)\).
The equality characterization in Theorem~\ref{thm:thirdmoment} shows that the representing chain
is an Erlang process with \(s\) phases.
 According to moment identities \eqref{eq:rkmomentsfour} and \eqref{eq:erlangmoments}, the common rate \(\kappa\) is determined by
\[
\kappa=\frac{s+1}{\rho_s+2}
=\frac{s+1-\sqrt{s^2-1}}{2}.
\]
The Erlang representation gives
\[
 \Green=\frac1\kappa
 \begin{pmatrix}
 1&&&\\
 1&1&&\\
 \vdots&\vdots&\ddots&\\
 1&1&\cdots&1
 \end{pmatrix},\qquad
 \phPi=(0,\ldots,0,1).
\]
The transform \eqref{eq:transform} gives
\[
 A=\frac{\Green-I}{\rho_s},\qquad
 b=\frac{\sigma}{\rho_s}(\phPi\Green)^\trans
   =\frac1s\one.
\]
Since 
\[\frac{1-\kappa}{\rho_s\kappa}=\frac12\left(1-\sqrt{\frac{s-1}{s+1}}\right)=\gamma,\qquad
\frac{1}{\rho_s\kappa}=\frac1{\sqrt{s^2-1}}=\delta,\] these are the coefficients in
\eqref{eq:thirdtableau}. Hence, any tableau attaining
\(R(A,b)=\rho_s\) coincides with \eqref{eq:thirdtableau}
after a simultaneous permutation of stages.
\end{proof}

\begin{remark}
The third-order upper bound and its equality characterization
use only the chain order conditions, through the 
moment identities in \eqref{eq:rkmomentsfour}.
The unique attaining tableau also satisfies the remaining
condition \[b^\trans(c\cdot c)=1/3.\]
On the one hand, it implies that this branching condition does
not sharpen the bound on the optimal radius of absolute monotonicity.
On the other hand, the tableau \eqref{eq:thirdtableau} remains
uniquely optimal in a larger class
of SSP Runge--Kutta methods, which have third-order accuracy for
linear problems.
\end{remark}

\section{Conclusion}
\label{sec:conclusion}
In this paper, we established a correspondence between SSP Runge--Kutta tableaux and
finite PH distributions. Under the absolute-monotonicity conditions, the
relation $\Green=I+rA$ links the coefficient matrix $A$ with the Green matrix
$\Green$, the output weights $b$ correspond to the initial distribution $\phPi$. Also, the chain order conditions can be transformed to the moments of absorption-time. 
We proved a sharp
third-moment ratio bound via state elimination and a nonnegative algebraic
decomposition and we established the general moment ratio bound for all acyclic PH distributions. Combining these moment bounds with the RK--PH
correspondence, we proved the conjectured optimal radii of absolute monotonicity $R_{s,2}=2s$
and $R_{s,3}=s-1+\sqrt{s^2-1}$ for general implicit tableaux, and
classified the unique attaining tableaux up to simultaneous stage permutations.

A possible direction for future work is to extend the present probabilistic framework to SSP two-step Runge--Kutta methods \cite{ketcheson2011strong}. Such an extension may involve CTMC models with two absorbing states, corresponding to the history values $u_n$ and $u_{n-1}$.

\section*{Declarations}
The authors used AI tools to assist with
proof development and manuscript drafting. The authors have verified all the mathematical arguments and take full responsibility for the content of this work.

\raggedbottom
\bibliographystyle{siamplain}
\bibliography{references}

\end{document}